\documentclass[11pt]{amsart}
\usepackage{amsmath, amsthm, amssymb,xspace}
\usepackage[breaklinks=true]{hyperref}%To get clickable links to papers
\usepackage{amsmath,amscd}
\usepackage{enumerate}
\usepackage{xcolor}

\newtheorem{theorem}{Theorem}
\newtheorem{lemma}[theorem]{Lemma}
\newtheorem{prop}[theorem]{Proposition}

\theoremstyle{definition}

\newtheorem{remark}[theorem]{Remark}
\newtheorem*{problem}{The Commuting Problem}

\newcommand{\disk}{\mathbb{D}}

\newcommand{\C}{\mathbb{C}}
\newcommand{\D}{\mathbb{D}}
\newcommand{\K}{\mathbb{K}}
\newcommand{\R}{\mathbb{R}}
\renewcommand{\H}{\mathbb{H}}

\begin{document}

\title{Commutants of Toeplitz Operators with Harmonic Symbols}

\author{Trieu Le}
\address{Department of Mathematics and Statistics, Mail Stop 942, University of Toledo, Toledo, OH 43606}
\email{Trieu.Le2@utoledo.edu}

\author{Akaki Tikaradze}
\address{Department of Mathematics and Statistics, Mail Stop 942, University of Toledo, Toledo, OH 43606}
\email{Akaki.Tikaradze@utoledo.edu}

\subjclass[2010]{Primary 47B35}
\keywords{Toeplitz operator; Bergman space; commuting problem}

\begin{abstract}

We make a progress towards describing the commutants of Toeplitz
operators with harmonic symbols on the Bergman space over the unit disk. Our work greatly generalizes several partial results in the field.

\end{abstract}

\maketitle

\section{Introduction}

Let $\disk$ denote the open unit disk
in the complex plane and $dA$ be the normalized Lebesgue area measure on
$\disk$. As usual, $L^2(\disk)$ is the space of complex-valued measurable functions $h$ on $\disk$
such that 
\[\|h\|_2=\Big(\int_{\disk}|h(z)|^2 dA(z)\Big)^{1/2} < \infty.\]
The Bergman space $A^2(\disk)$ is the closed subspace
of $L^2(\disk)$ consisting of holomorphic functions. Let $P : L^2(\disk)\to A^2(\disk)$ denote
the orthogonal projection. For a function $f\in L^{\infty}(\disk),$ one defines the Toeplitz operator $T_f : A^2(\disk)\to A^2 (\disk)$ as 
\[T_f(h) = P(fh)\ \text{ for all } \ h\in A^2(\disk).\]
It is clear that $T_{f}$ is bounded with $\|T_{f}\|\leq\|f\|_{\infty}$. 

Various algebraic problems related to Toeplitz operators have been extensively studied in the literature (see, e.g. \cite{ZhuAMS2007} and the references therein). One of the most interesting problems in the field is the commuting problem. Recall that two operators $A$ and $B$ commute if and only if their commutator $[A,B]=AB-BA$ vanishes.

\begin{problem}
Given a non-constant function $f\in L^{\infty}(\disk)$, determine all $g\in L^{\infty}(\disk)$ such that $[T_g,T_{f}]=0$ on the Bergman space.
\end{problem}

This problem was motivated by the same problem for Toeplitz operators on the Hardy space over the unit disk, which was solved by Brown and Halmos in their seminal paper \cite{Brown1963}. 
On the Bergman space, the Commuting Problem has been completely solved when $f$ is holomorphic or radial. More specifically, Axler, {\v{C}}u{\v{c}}kovi{\'c} and Rao \cite{ACR} showed
\begin{theorem}
\label{T:ACR}
Let $f$ be a non-constant bounded holomorphic function.  If $[T_{g},T_{f}]=0$ on $A^2(\disk)$, then $g$ must be holomorphic as well.
\end{theorem}

Recall that a function $f$ is radial if there exists a function $\varphi$ defined on $[0,1)$ such that $f(z)=\varphi(|z|)$ for a.e. $z\in\disk$. Among other things, {\v{C}}u{\v{c}}kovi{\'c} and Rao \cite{CR} proved that if $f$ is non-constant radial and $[T_{g},T_{f}]=0$, then $g$ must be radial as well.

Besides these complete results, only partial answers to the Commuting Problem have been known for other cases. Axler and {\v{C}}u{\v{c}}kovi{\'c} \cite{AC} offered the following result, which we state here in a slightly different form from the original statement.
\begin{theorem}
\label{T:AC}
Let $f$ be a bounded harmonic function which is not holomorphic or conjugate holomorphic. If $g$ is bounded harmonic and $[T_g,T_f]=0$ on $A^2(\disk)$, then $g=af+b$ for some $a,b\in\C$.
\end{theorem}
The problem is still wide open when $f$ is assumed to be harmonic but the harmonicity of $g$ is dropped. There have been several partial results in this direction. In order to describe those results, we need the decomposition of $L^2(\disk)$ as an orthogonal direct sum of homogeneous pieces (see, e.g. \cite{CR}). Given a function $g\in L^{2}(\disk)$, we may write 
\begin{align}
\label{Eqn:polar_decomp}
g(z)=g(re^{i\theta})=\sum_{k=-\infty}^{\infty}g_k(r)e^{ik\theta}
\end{align}
for $z=re^{i\theta}\in\disk$.
The series converges in the norm of $L^2(\disk)$ and for each integer $k$, the function $g_k$ is defined a.e. on $[0,1)$ by
\begin{align}
\label{Eqn:g_k}
g_k(r) = \frac{1}{2\pi}\int_{0}^{2\pi}g(re^{i\theta})\,e^{-ik\theta}\,d\theta.
\end{align}
Note that if $g$ is bounded on $\disk$, then all $g_k$ are bounded on $[0, 1)$. 

{%\color{red}
For a function $g$ in $L^1(\disk)$, the integral in \eqref{Eqn:g_k} still exists and defines $g_k$ a.e. on $[0,1)$ for each integer $k$. In fact, all $g_k$ belongs to $L^1([0,1),rdr)$. We shall call $g\in L^1(\disk)$ a \textit{right-terminating} function if there exists an integer $n$ such that 
\[g(re^{i\theta}) = \sum_{k=-\infty}^{n}g_k(r)e^{ik\theta},\]
where the series converges in $L^1(\disk)$.}

It was shown in \cite{LRY} that if $f(z)=z+\bar{z}$ and $g$ is a right-terminating bounded function such that $[T_g,T_f]=0$, then the conclusion of Theorem \ref{T:AC} holds. That is, $g=af+b$ for some complex constants $a$ and $b$. In his thesis \cite[Theorem 3.4.3]{Y}, Yousef improved this result by considering $f$ of the form $f(z)=z+\bar{\phi}(z)$, where $\phi$ is holomorphic and $\phi'(0)\neq 0$. 

Recall that the Bergman space $A^2(\disk)$ is a reproducing kernel Hilbert space with kernel $K(z,w)=(1-\bar{w}z)^{-2}$. It follows that $T_{f}$ is an integral operator. Indeed, for $h\in A^2(\disk)$, we have
\begin{align}
\label{Eqn:int_formula_TO}
T_f(h)(z) & = \int_{\D}\frac{f(w)h(w)}{(1-\bar{w}z)^2}\,dA(w)\ \text{ for } z\in\disk.
\end{align}
When $f$ is unbounded but belongs to $L^1(\disk)$, the integral on the right-hand side of \eqref{Eqn:int_formula_TO} may still exist, for example, when $h$ is bounded. One then uses \eqref{Eqn:int_formula_TO} to define the Toeplitz operator $T_{f}$ for $L^1$-symbol $f$. Such operator may not be bounded on $A^2(\disk)$ in general. However, if $f$ is supported on a compact subset of the unit disk, then it can be verified that $T_{f}$ is not only bounded but also compact. Consequently, if $f$ belongs to $L^1(\disk)$ and is bounded on an annulus $\{z: r<|z|<1\}$, for some $0<r<1$, then $T_{f}$ is bounded.

A calculation in \cite[Lemma 2]{LRY} shows that
\begin{align*}
   (T_{z+\bar{z}})^2 & = T_{g_1} \text{ and } (T_{z+\bar{z}})^3 = T_{g_2},
\end{align*}
for some unbounded functions $g_1, g_2$ in $L^1(\D)$. We note that \[g_1(z)=z^2+\bar{z}^2+|z|^2+\ln|z|^2+1\] but the formula for $g_2$ is more complicated. Consequently, $T_{g_1}$ and $T_{g_2}$ commute with $T_{f}$ (for $f(z)=z+\bar{z}$) but neither $g_1$ nor $g_2$ can be written in the form $af+b$ for some constant $a,b$. Furthermore, \cite[Theorem 2]{LRY} states that for a right-terminating function $g\in L^1(\disk)$ such that $T_{g}$ is bounded, we have $[T_g,T_{f}]=0$ if and only if $g\in{\rm span}(\{1, f, g_1, g_2\})$.

With the same fashion, the discussion in \cite[Remark 4]{LRY} shows that $(T_{z+\bar{z}^2})^2 = T_{h}$, where $h$ is an unbounded function in $L^1(\disk)$. A similar approach as in \cite[Theorem 2]{LRY} shows that if $g$ is a right-terminating function in $L^1(\disk)$ and $[T_g,T_{z+\bar{z}^2}]=0$, then $g\in{\rm span}(\{1,z+\bar{z}^2,h\})$.

Alsabi and Louhichi in a recent preprint \cite{AL} considered the case $f(z)=z^2+\bar{z}^2$ and $g$ is a right-terminating function in $L^1(\disk)$. Their result asserts that the conclusion of Theorem \ref{T:AC} still holds.

In the present note we study the case $f$ belongs to a certain large class of bounded harmonic functions. Our approach generalizes that of \cite{LRY} and \cite[Section 3.4]{Y}. Our main result is the following.

\begin{theorem}
\label{main}
Let $f$ be a bounded harmonic function of the form \[f(z)=f_1(z)+\bar{z}^{m}\,\overline{f_2(z)},\] where $m$ is a positive integer, $f_1$ is a non-constant holomorphic polynomial and $f_2$ is a holomorphic function on $\disk$ with $f_2(0)\neq 0$. Let $g\in L^1(\disk)$ be a right-terminating function such that $T_g$ is bounded and $[T_g,T_f]=0$ on $A^2(\disk)$. If either $m\geq 3$ or $g$ is bounded, then $g=af+b$ for some $a, b\in \mathbb{C}$.
\end{theorem}

\begin{remark}
Given $f$ a right-terminating bounded harmonic function which is not holomorphic or conjugate holomorphic, Theorem \ref{main} shows that if $g$ is a right-terminating \textit{bounded} function and $T_g$ commutes with $T_{f}$, then $g$ must be a linear combination of $f$ and a constant. In the case $g$ is allowed to be unbounded, the examples involving $z+\bar{z}$ and $z+\bar{z}^2$ discussed earlier shows that the condition $m\geq 3$ is indeed necessary.
\end{remark}

\section{The main result}
 
Recall that for all integers $p$ and non-negative integers $m$, we have the formulas (see, e.g. \cite[Lemma 1]{LRY})
\begin{align}
\label{Eqn:T_formulaA}
T_{\bar{z}^m}(z^k) & = 
\begin{cases}
\frac{k-m+1}{k+1}z^{k-m} & \text{ if } k\geq m-1\\
0 & \text{ if } 0\leq k \leq m-2
\end{cases},
\end{align}
and 
\begin{align}
\label{Eqn:T_formulaB}
T_{e^{-ip\theta} \phi(r)}(z^k) & = 
\begin{cases}
(2n-2p+2)\widehat{\phi}(2k-p+2)z^{k-p} & \text{ if } k\geq \max\{0,p\}\\
0 & \text{ if } 0\leq k\leq p-1
\end{cases},
\end{align}
where $\widehat{\phi}(\zeta)=\int_{0}^{1} \phi(r)r^{\zeta-1}dr$ is the Mellin transform of $\phi$. 

Let $\H$ denote the right-half plane which consists of all complex numbers with positive real parts:
\[\H=\{\zeta\in\C: \Re(\zeta)>0\}.\] We shall write $\overline{\H}$ for the closure of $\H$ as a subset of the complex plane. 

For $\phi\in L^1([0,1),rdr)$, it follows from Fubini's and Morera's Theorems that the Mellin transform $\widehat{\phi}$ is holomorphic on $2+\H$, the half-plane consisting of all complex numbers $\zeta$ with $\Re(\zeta)>2$. Furthermore, $\widehat{\phi}$ is continuous and bounded on $2+\overline{\H}$.

On the other hand, for $\phi$ bounded, $\widehat{\phi}$ is holomorphic on $\H$ and for all $\zeta\in\H$, we have the estimate
\begin{align}
\label{Eqn:Mellin_estimate}
|\widehat{\phi}(\zeta)| \leq \frac{\|\phi\|_{\infty}}{\Re(\zeta)}.
\end{align}

We recall the following classical result \cite[p.102]{Rem} on the zero set of polynomially bounded holomorphic functions on $\H$. It was originally proved for bounded holomorphic functions but by replacing $h(z)$ by $h(z)/(z+1)^N$ for some sufficiently large integer $N$, the result remains true as stated.

\begin{lemma}
\label{L:uniquiness_set}
Let $h$ be a polynomially bounded holomorphic function on $\H$ that vanishes at the pairwise distinct points $a_1, a_2, \ldots$. If $\inf_{j\geq 1}\{|a_j|\} > 0$ and $\sum_{j=1}^{\infty}\Re(1/a_j) = \infty$, then $h$ is identically zero.
\end{lemma}

We shall also make use of the following two facts. Since we have not been able to locate appropriate references, we sketch the ideas of the proofs.

\begin{lemma}
\label{L:identity_result}
Let $G$ be a holomorphic function on an open connected set $\Omega\subset\C$ and $F$ be a meromorphic function on $\C$. Suppose that $G=F$ on a non-empty open subset of $\Omega$. Then the poles of $F$ lie outside of $\Omega$ and $G=F$ on $\Omega$. 
\end{lemma}
\begin{proof}
Write $F=g/f$, where $f$ and $g$ are entire that do not have common zeroes. The hypothesis and the identity theorem shows that $fG=g$ on $\Omega$. The conclusion then follows.
\end{proof}

\begin{lemma}
\label{L:preriodic_hol}
Let $u$ be a polynomially bounded holomorphic function on a right-half plane (i.e., a translation of $\H$ by a real number). Suppose there exists $r>0$ such that $u(k+r)=u(k)$ for all sufficiently large integers $k$. Then $u$ is a constant function.
\end{lemma}
\begin{proof}
Lemma \ref{L:uniquiness_set} implies that $u(\zeta+r)=u(\zeta)$ for all $\zeta$ belonging to the domain of $u$. It then follows that $u$ extends to an entire function $\tilde{u}$ with period $r$. Since $u$ is polynomially bounded, $\tilde{u}$ is polynomially bounded as well. Consequently, $\tilde{u}$, hence $u$, is constant.
\end{proof}

The following proposition will play a key role in the proof of our main result.
 
\begin{prop}
\label{key} 
Suppose $p$ is an arbitrary integer and $l,m,n$ are positive integers. Suppose $\phi$ is a function in $L^1([0,1),rdr)$ such that
\begin{align*}
[T_{\phi(r)e^{-ip\theta}} , T_{z^l}]=[T_{\bar{z}^m}, T_{z^n}].
\end{align*}
Then one of the following statements hold.
\begin{enumerate}[(a)]
   \item $p=m$ and $l=n$.
   \item $p\in\{0,1\}$ and $m=p+1$.
   \item $p=-1$ and $m=1$. 
\end{enumerate}
If (b) or (c) holds, then $m\leq 2$ and $\phi$ is unbounded on $[0,1)$. 
\end{prop}

\begin{proof}
A direct calculation using \eqref{Eqn:T_formulaA} gives
\begin{align}
\label{Eqn:comm_A}
[T_{\bar{z}^m},T_{z^n}](z^k)  = \begin{cases}
\frac{mn}{(k+1)(k+n+1)}\,z^{k+n-m} & \text{ if } k\geq m\\
\frac{k+n-m+1}{k+n+1}\,z^{k+n-m} & \text{ if } m-n\leq k \leq m-1.\\
0 & \text{ otherwise}
\end{cases}
\end{align}
Similarly, using \eqref{Eqn:T_formulaB}, we obtain
\begin{align}
\label{Eqn:comm_B}
[T_{\phi(r)e^{-ip\theta}} , T_{z^l}](z^k) 
= \begin{cases}
\big(G(k+l)-G(k)\big)z^{k+l-p} & \text{ if } k\geq p\\
G(k+l)\,z^{k+l-p} & \text{ if } p-l\leq k\leq p-1.\\
0 & \text{ otherwise}
\end{cases}
\end{align}
Here, \[G(\zeta)=(2\zeta-2p+2)\widehat{\phi}(2\zeta-p+2).\]
Since $\phi\in L^1([0,1),rdr)$, the function $G$ is holomorphic on the half plane $\K=\{\zeta:\Re(\zeta)>p/2\}=p/2+\H$ and it is continuous and polynomially bounded on $\overline{\K}$. 

Because the two commutators are equal, $p-l=m-n$ and we have
\begin{equation}
\label{Eqn:diff_eqn_G}
G(\zeta+l)-G(\zeta)=\frac{mn}{(\zeta+1)(\zeta+n+1)} \ \text{ for all } \zeta=k\geq\max\{m,p\}.
\end{equation}
Lemma \ref{L:identity_result} shows that \eqref{Eqn:diff_eqn_G} holds for all $\zeta\in\overline{\K}$. Define
\[F(\zeta) = -\sum_{j=0}^{\infty}\frac{mn}{(\zeta+jl+1)(\zeta+jl+n+1)}.\]
Then $F$ is a meromorphic function on the complex plane with poles a subset of the negative integers and for all integers $k\geq\max\{m,p\}$,
\[F(k+l)-F(k) = \frac{mn}{(k+1)(k+n+1)} = G(k+l)-G(k).\]
Since $G-F$ is holomorphic and polynomially bounded on $\K\cap\H$, Lemma \ref{L:preriodic_hol} implies that $G-F$ is a constant function on this set. Therefore, there is a constant $c$ such that for all $\zeta\in\K\cap\H$,
\begin{align}
\label{Eqn:formula_G}
G(\zeta)=c-\sum_{j=0}^{\infty}\frac{mn}{(jl+\zeta+1)(jl+\zeta+n+1)}.
\end{align}
Lemma \ref{L:identity_result} shows that the above identity in fact holds for all $\zeta\in\overline{\K}$. In particular, the pole $\zeta=-1$ of the right-hand side of \eqref{Eqn:formula_G} does not lie in $\overline{\K}$. It follows that $-1<p/2$, which gives $p>-2$, hence, $p\geq -1$ since it is an integer. Note that \eqref{Eqn:formula_G} also shows that $G$ is an increasing function on $\K\cap\R$.

If $p-1\geq m$, then setting $k=m-1$ (respectively, $k=m$) in \eqref{Eqn:comm_A} and \eqref{Eqn:comm_B} gives
$G(m+l-1) = \frac{n}{n+m}$ (respectively, $G(m+l) = \frac{mn}{(m+1)(m+n+1)}$). It follows that $G(m+l-1)>G(m+l)$, which contradicts the fact that $G$ is increasing on $\K\cap\R$. Consequently, $p\leq m$. 

Our argument so far has shown $-1\leq p \leq m$. Suppose $p\geq 2$. Then $p-1$ belongs to $\overline{\K}$, which implies that $G$ is continuous at $p-1$. The defining formula for $G$ gives $G(p-1)=0$. This together with \eqref{Eqn:diff_eqn_G} shows that $G(p+l-1)=mn/(p(p+n))$. On the other hand, setting $k=p-1$ in \eqref{Eqn:comm_A} and \eqref{Eqn:comm_B} gives $G(p+l-1) = \frac{p+n-m}{p+n}$. Consequently,
\[\frac{mn}{p(p+n)} = \frac{p+n-m}{p+n},\]
which implies $p=m$. It follows that (a) holds, since $p-l=m-n$.

We now suppose that (a) does not hold. Then $-1\leq p\leq\min\{1, m-1\}$. Consider first the case $p\geq 0$ (that is, $p=0$ or $p=1$). Then $p\in\overline{\K}$ and hence, by \eqref{Eqn:diff_eqn_G},
\[G(p+l)-G(p) = \frac{mn}{(p+1)(p+n+1)}.\] 
Setting $k=p$ in \eqref{Eqn:comm_A} and \eqref{Eqn:comm_B} gives $G(p+l)-G(p)=(p+n-m+1)/(p+n+1)$. It follows that
\[\frac{mn}{(p+1)(p+n+1)} = \frac{p+n-m+1}{p+n+1},\]
which forces $p=m-1$. Therefore, (b) holds.

Now consider $p=-1$. Since $0\in\K$, \eqref{Eqn:diff_eqn_G} gives
$G(l)-G(0) = mn/(n+1)$. On the other hand, setting $k=0$ in \eqref{Eqn:comm_A} and \eqref{Eqn:comm_B} gives
$G(l)-G(0) = (n-m+1)/(n+1)$. We then have
\[\frac{mn}{n+1} = \frac{n-m+1}{n+1},\]
which implies $m=1$. Therefore, (c) holds.

To complete the proof, we assume that $\phi$ is bounded and show that neither (b) or (c) may occur. The boundedness of $\phi$ implies that $\widehat{\phi}$ is holomorphic on $\H$, which implies that $G$ is holomorphic on $(p/2-1)+\H$. By \eqref{Eqn:formula_G}, the pole $\zeta=-1$ of the right-hand side must lie outside of $(p/2-1)+\H$, which gives $p\geq 0$. 

If $p=0$, then \eqref{Eqn:formula_G} and the defining formula for $G$ give
\[(2\zeta+2)\widehat{\phi}(2\zeta+2) = c - \sum_{j=0}^{\infty}\frac{mn}{(jl+\zeta+1)(jl+\zeta+n+1)}.\]  This shows that $\lim_{t\to -1^{+}}(2t+2)\widehat{\phi}(2t+2)=-\infty$, which contradicts \eqref{Eqn:Mellin_estimate}. 

If $p\geq 1$, then $G$ is holomorphic on $(p/2-1)+\H$, which contains $p-1$. The argument that we have seen earlier shows that $m=p$. In any case, neither (b) nor (c) may happen. This completes the proof of the proposition.
\end{proof}

We are now ready for the proof of our main result.

\begin{proof}[Proof of Theorem {\ref{main}}]
Assume that $g$ is not constant and $[T_{g},T_{f}]=0$. By taking adjoints, we also have $[T_{\bar{g}},T_{\bar{f}}]=0$. Since neither $f$ nor $\bar{f}$ is holomorphic, Theorem \ref{T:ACR} implies that neither $g$ or $\bar{g}$ is holomorphic. We write
\[g(re^{i\theta}) = \sum_{j=-\infty}^{n}g_j(r)e^{ij\theta}\] for some integer $n$ and $g_n$ is not identically zero. Write $f_1(z)=z^{l}+\textit{lower powers of } z$, for some integer $l\geq 1$ and
$f_2(z) = c+\textit{higher powers of } z$ for some $c\neq 0$.
By considering the highest order terms in the identity $[T_{g},T_{f}]=0$, we conclude that $[T_{g_n(r)e^{in\theta}}, T_{z^l}]=0$.
This (by Theorem \ref{T:ACR} again) implies that $n\geq 0$ and $g_n(r)e^{in\theta}=a z^{n}$ for some $a\neq 0$. Let $s$ be the largest integer such that $g_{s}(r)e^{is\theta}$ is not a holomorphic monomial. Since $g$ is not holomorphic, such $s$ exists. We write $g=\psi_1+\psi_2$, where
\[\psi_1(z)=a z^{n}+\textit{lower powers of $z$}\] is a holomorphic polynomial and \[\psi_2(re^{i\theta})=g_s(r)e^{is\theta}+\sum_{k=-\infty}^{s-1}g_k(r)e^{ik\theta}. \]
If $n=0$, then $\psi_1$ is constant and hence $[T_{\psi_2},T_{f}]=0$. The above argument will imply that $g_s(r)e^{is\theta}$ is 
holomorphic, which contradicts our choice of $s$. Therefore, $n\geq 1$. 

Since $[T_{\psi_1},T_{f_1}]=0$, we obtain
\begin{align*}
0 = [T_{g},T_{f}] = [T_{g},T_{f_1}]+[T_{g},T_{\bar{z}^m\bar{f_2}}] = [T_{\psi_2},T_{\varphi_1}]+[T_{g},T_{\bar{z}^m\bar{f_2}}].
\end{align*}
Considering the highest order terms again, we conclude that that either $[T_{g_s(r)e^{is\theta}}, T_{z^l}]+ac[T_{z^n}, T_{\bar{z}^m}]=0$, or $[T_{g_s(r)e^{is\theta}}, T_{z^l}]=0$, or $[T_{z^n}, T_{\bar{z}^m}]=0$. Since neither of the last two equalities can
happen (because $n,m,l$ are all positive and $g_s(r)e^{is\theta}$ is not holomorphic), the first equality must hold. By  Lemma \ref{key}, 
we have $n=l$. 
Thus, we have shown that if $g$ is not a constant function and $[T_{g},T_{f}]=0$, then $g$ has the form
$g(z)=az^{l}+\sum_{k=-\infty}^{l-1}g_{k}(r)e^{ik\theta}$ for $z=re^{i\theta}$, where $a$ is a non-zero constant. Now put $\tilde{g}=g-a f$. Then $[T_{\tilde{g}},T_{f}]=0$ and $\tilde{g}$ does not admit the above representation. It follows that $\tilde{g}$ must be a constant function. Therefore, $g=af+b$ for $a,b\in\C$ as required.
\end{proof}

\providecommand{\MR}{\relax\ifhmode\unskip\space\fi MR }
% \MRhref is called by the amsart/book/proc definition of \MR.
\providecommand{\MRhref}[2]{%
  \href{http://www.ams.org/mathscinet-getitem?mr=#1}{#2}
}
\providecommand{\href}[2]{#2}

\end{document}